\documentclass[letterpaper]{article}

\usepackage{packages-styles/general-math}

\usepackage{packages-styles/general-text}

\usepackage{xcolor}
\usepackage{arxiv}

\hypersetup{
pdftitle={Non-local Optimization: Imposing Structure on Optimization Problems by Relaxation},
pdfsubject={Mathematical Optimization},
pdfauthor={Nils Mueller, Tobias Glasmachers},
pdfkeywords={Global optimization, Stochastic optimization, Evolution strategies, Robust optimization}
}

\definecolor{tgcolor}{rgb}{0.2,0.6,0.8}
\definecolor{nmcolor}{rgb}{0., .7, 0.}

\title{Non-local Optimization: Imposing Structure on Optimization Problems by Relaxation}

\author{
{\normalfont
\begin{tabular}[t]{c@{\extracolsep{8em}}c} 
{\bfseries Nils M{\"u}ller} & {\bfseries Tobias Glasmachers} \\
{\itshape Ruhr University Bochum} & {\itshape Ruhr University Bochum} \\ 
{\itshape Germany} & {\itshape Institute for Neural Computation} \\
 & {\itshape Germany} \\
 \href{mailto:mail@nilsmueller.io}{\texttt{mail@nilsmueller.io}} & \href{tobias.glasmachers@ini.rub.de}{\texttt{tobias.glasmachers@ini.rub.de}}
\end{tabular}
}
}

\begin{document}
\maketitle

\begin{abstract}
      In stochastic optimization, particularly in evolutionary computation and reinforcement learning, the optimization of a function $f: \Omega \to \mathbb{R}$ is often addressed through optimizing a so-called relaxation $\theta \in \Theta \mapsto \E_\theta(f)$ of $f$, where $\Theta$ resembles the parameters of a family of probability measures on $\Omega$. We investigate the structure of such relaxations by means of measure theory and Fourier analysis, enabling us to shed light on the success of many associated stochastic optimization methods. The main structural traits we derive and that allow fast and reliable optimization of relaxations are the consistency of optimal values of $f$, Lipschitzness of gradients, and convexity. We emphasize settings where $f$ itself is not differentiable or convex, e.g., in the presence of (stochastic) disturbance.
\end{abstract}

\keywords{Global optimization \and Stochastic optimization \and Evolution strategies \and Robust optimization}

\section{Introduction}

The idea of an optimization problem is to find minima (or maxima) of a function $f: \Omega \to \R$ in a set of possible candidates $\Omega$. Given a $\sigma$-algebra $\mathcal{A}$ on $\Omega$ we can understand a probability measure $\Ps{}$ on $(\Omega, \mathcal{A})$ as a generalized candidate of $f$. A generalized candidate $\Ps{}$ induces a functional to which we assign its expected value $\E(f)$. A parameterization $\Ps{\cdot}: \Theta \to \mathcal{M}_1^+(\Omega, \mathcal{A})$ of (a subset of) the generalized candidates, where $\Theta$ is the parameter set and $\mathcal{M}_1^+(\Omega, \mathcal{A})$ is the set of probability measures on $(\Omega, \mathcal{A})$, results in what is called a \textbf{stochastic relaxation} $\theta \in \Theta \mapsto \E_\theta(f)$.
In this paper, we investigate the central structural properties of relaxed problems that are of interest to optimization: consistency, smoothness, and convexity.

\paragraph{Algorithms.}
Our work applies to a broad range of optimization algorithms, including several randomized search heuristics (variants of estimation of distribution methods) for discrete and continuous search spaces. The second half of our paper is heavily influenced by applications in gradient-based optimization, covering in particular information geometric optimization \cite{ollivier2017information}.
The two most important use-cases in the regime of gradient-based optimization are the following:
\begin{itemize}
    \item $\Theta$ is an open subset of $\R^n$ and the measures have $\theta$-differen\-tiable densities with respect some measure $\mu$, i.e. $\Ps{\theta} = k_\theta \mu$.
    \item $\Omega=\Theta=\R^n$, $\mathcal{A}$ is the Borel $\sigma$-algebra of $\R^n$, i.e., $\mathcal{A} = \mathcal{B}(\R^n)$, $f$ is differentiable, and $\Pr_\theta \circ \psi_\theta = \Ps{0}$,
    where $\psi_\theta: \R^n \to \R^n, x \mapsto x+\theta$.
\end{itemize}
Under mild regularity assumptions, which we will cover, we have the identity
\[
   \nabla_\theta \E_\cdot(f) \equiv \E_\cdot(f \nabla_\theta \ln k_\cdot)
\]
in the first case and
\[
   \nabla_\theta \E_\cdot(f) \equiv \E_\cdot( \nabla_x f)
\]
in the second. In both cases, the gradient of the relaxation can, therefore, be approximated by numerical integration.

\paragraph{Motivation and Related Work.}
Stochastic relaxations have long provided a powerful approach to optimization problems where non-local structure is significant, problem representations are not easy to manipulate (e.g., through differentiation), or robust solutions are desired. In recent literature, there has been considerable progress on qualitative and quantitative assessment of optimization methods of such relaxations' parameter spaces, particularly concerning their convergence on problems where local structure is instructive \cite{hgs2014igo, nesterov2017random, zhang2018rie, maggiar2018trustregion}. Even before, there have been numerous proposals of frameworks that guide the principled design of optimization methods for a wide range of discrete and continuous problems \cite{wierstra2008nes, ollivier2017information, malago2009sr}. The use of stochastic relaxations, however, has been motivated by invariance of solutions under transformations \cite{ollivier2017information}, biological plausibility of evolutionary computation \cite{wierstra2008nes}, program simplicity, as well as practical utility \cite{nesterov2017random, salimans2017evolution}, or by experiment \cite{choro2019asebo, salimans2017evolution}. Yet the favorable structure of stochastic relaxations for optimization itself has received a rather incidental treatment in favor of different questions. Outstanding from the rest of the literature, \cite{mobahi2012gaussian,mobahi2015theoretical,loog2001critical} derive insightful and principled properties of special relaxations under decay and convexity conditions that arguably deviate from the classical setting in optimization. Furthermore, deep connections of Gaussian relaxations and approximation have been established in \cite{mobahi2015link}. Surprisingly, much of the referenced prior work seems to have been compiled in virtually independent communities.\\

In this work, we investigate the relation between a problem and its stochastic relaxations and derive a systematic understanding of why stochastic relaxations are favorable for fast and reliable optimization in practice.

\begin{figure}[!ht]
\makebox[\textwidth][c]{
    \begin{tikzpicture}
        \begin{axis}[scale only axis=true,axis lines=none,at={(0,0)},xmin=-4.2,xmax=4.2,ymin=-2.5,ymax=2.5,height=7.cm,width=11.76cm]
            
            \filldraw[line width=2.0pt, draw=black, fill=white] (-4,-1.5) rectangle (-1,1.5);
            
            \foreach \i in {1,2,3,4,...,112}{
                \addplot[shift={(axis direction cs:-2.5,0)},black,->-=0.5] table[]{data/data-non-conv/rg\i.txt};
            }

        
            \node[anchor=north] at (-2.5,-1.65) {\large $\nabla f$};

            
            \filldraw[line width=2.0pt, draw=black,fill=white] (1,-1.5) rectangle (4,1.5);
            \foreach \i in {1,2,3,4,...,64}{
                \addplot[shift={(axis direction cs:2.5,0)},black,->-=0.5] table[]{data/data-conv/rg\i.txt};
            }
            \node[anchor=north] at (2.5, -1.65) {\large $\nabla \E_{\cdot}(f)$};
            \path[ultra thick, ->,>=latex] (-1,0) edge  node[above, yshift=0.3] {\large $\E_\cdot^*$} (1,0);
        \end{axis}
    \end{tikzpicture}
    }
    \caption{Gradient flows of the Rastrigin function and its isotropic Gaussian relaxation.}
\end{figure}
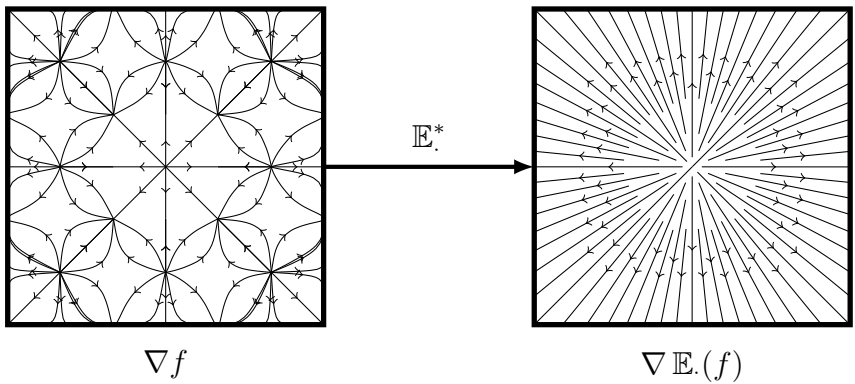

\paragraph{Outline.}
In \cref{sec:Transforming an Optimization Problem}, we showcase the structure of a stochastic relaxation of the Rastrigin function, a popular highly multi-modal benchmark problem in optimization. We also introduce the general problem definition covering discrete and continuous spaces. In \cref{sec:Consistency in the Codomain}, a criterion for the consistency of optimal function values and the location of optima is established. In \cref{sec:Imposing Differentiability}, we derive representations of derivatives of stochastic relaxations. Among this we cover the transfer of (Lipschitz) differentiability from parameterized densities to the stochastic relaxation. Here we tie together many notions of differentiability that are used in practice. In \cref{sec:Weak Convexity}, we derive insightful sufficient conditions under which functions on $\R^n$ have a convex stochastic relaxation based on a notion of weak convexity and Fourier analysis.

\section{Transforming an Optimization Problem}
\label{sec:Transforming an Optimization Problem}

In this section, we will demonstrate by example, what we will establish in later sections for more general cases and what can be understood as the driving factors for the success of stochastic relaxations. We chose a popular benchmark problem from optimization literature, the Rastrigin function, and multivariate isotropic normal distributions as parameterized probability measures. This setting is typical in stochastic optimization, particularly with evolution strategies. More concretely, we want to illustrate the following traits of this particular stochastic relaxation:
\begin{itemize}
    \item Consistency of the minimum of the original problem,
    \item Lipschitz continuity of gradients, and
    \item convexity.
\end{itemize}
The above traits will serve to guide our intuition in the subsequent sections.

\begin{example}
\label{ex:Rastrigin}
Consider the Rastrigin function
\[
    x \in \R^n \longmapsto f(x) := \norm{x}^2 - \sum_{j=1}^n a_j \cos( \xi_j x_j) \,,
\]
where $\xi_j, a_j > 0$ for all $j \in \N_{\leq n}$.
Let $\Ps{\theta, \sigma} \sim \mathcal{N}(\theta, \sigma^2 I)$, where $\theta \in \R^n$ and $\sigma > 0$. Rastrigin's function is an example of a highly multi-modal optimization problem. Its global minimum is found at the origin, while an exponentially large (in the problem dimension) number of local minima represent distractors to local optimization methods.
\\

We will now look at the structure of the function
\[
    \theta \in \R^n \longmapsto \E_{\theta, \sigma}(f) \,,
\]
which we interpret as a relaxation on the original candidate space $\R^n$ of $f$. We find the following traits.
\begin{itemize}
    \item We can pick a small $\sigma$ such that we can approximate the value of $f(x)$ with $\E_{x, \sigma}(f)$ with arbitrary precision---in particular, at the global minimum of $f$ at $0$.
    \item Structural traits such as the smoothness of the density of the normal distribution are transferred to $\E_{\cdot}(f)$. Moreover, the gradients of $\theta \mapsto \E_{\theta, \sigma}(f)$ are Lipschitz for all $\sigma > 0$.
    \item As the following arguments will show, we can pick $\sigma^*$ such that $\theta \mapsto \E_{\theta, \sigma}(f)$ is strictly convex for all $\sigma > \sigma^*$.
\end{itemize}
In our specific example we obtain the traits in terms of handy closed-form formulas:
\begin{enumerate}
\item Let $\lambda$ denote the Lebesgue measure of $\R^n$. For all $\theta \in \R^n$ and $\sigma > 0$, we have
\begin{align*}
    \E_{\theta, \sigma}(\norm{\cdot}^2)
    &=
    \int_{\R^n} \sum_{k=1}^n x_k^2 \tfrac{1}{\sqrt{(2 \pi \sigma^2 )^n}}\exp\left( -\tfrac{1}{2 \sigma^2}\norm{\theta- x}^2 \right) \lambda(\mathrm{d}x)
\tag*{\small{\text{(definitions)}}}
    \\
    &=
    \sum_{k=1}^n \int_{\R^n} x_k^2 \tfrac{1}{\sqrt{(2 \pi \sigma^2 )^n}}\exp\left( -\tfrac{1}{2 \sigma^2} \norm{\theta- x}^2 \right) \lambda(\mathrm{d}x)
\tag*{\small{\text{(linearity)}}}
    \\
    &=
    \sum_{k=1}^n \int_{\R} x_k^2 \tfrac{1}{\sqrt{2 \pi \sigma^2}}\exp\left( -\tfrac{1}{2 \sigma^2} (\theta_k- x_k)^2 \right) \lambda(\mathrm{d}x_k)
\tag*{\small{\text{(Fubini, linearity, and probability measure)}}}
    \\
    &=
    n\sigma^2 + \norm{\theta}^2 \,.
\tag*{\small{\text{(Steiner's Theorem)}}}
\end{align*}

\item Let $j \in \N_{\leq n}$ and define $\Xi_j := \xi_j e_j$, then, using the representation of the characteristic function of the multivariate normal distribution in the third equality, we obtain
\begin{align*}
    \E_{\theta, \sigma}(\cos(\xi_j (\cdot)_j))
    &=
    \int_{\R^n} \mathrm{Re} \left( e^{i\langle \Xi_j, x \rangle} \right) \, \Ps{\theta, \sigma}(\mathrm{d}x)
\tag*{\small{\text{(definitions)}}}
    \\
    &=
    \mathrm{Re} \bigg( \int_{\R^n} e^{i\langle \Xi_j, x \rangle}  \, \Ps{\theta, \sigma}(\mathrm{d}x)\bigg)
\tag*{\small{\text{(definition of the complex integral)}}}
    \\
    &=
    \mathrm{Re} \left( \exp\left( i \langle \Xi_j, \theta \rangle - \tfrac{1}{2} \sigma^2 \norm{\Xi_j}^2 \right) \right)
\tag*{\small{\text{(characteristic function of the multivariate normal)}}}
    \\
    &=
    \mathrm{Re}\left(\exp( i \xi_j\theta_j)\right) \exp\big(- \tfrac{1}{2} \sigma^2 \xi_j^2\big)
\tag*{\small{\text{(definitions and exponential rule)}}}
    \\
    &=
    \cos(\xi_j\theta_j)\exp\big(- \tfrac{1}{2} \sigma^2 \xi_j^2\big) \,.
\tag*{\small{\text{(definition of $\cos$)}}}
\end{align*}
\end{enumerate}
Therefore, by linearity of the integral, we have
\begin{align*}
    \E_{\theta, \sigma}(f)
    &=
    n\sigma^2 + \norm{\theta}^2 - \sum_{j=1}^n a_j \cos(\xi_j\theta_j)\exp\big(- \tfrac{1}{2} \sigma^2 \xi_j^2\big) \,.
\end{align*}
Now, we can check strict convexity by looking at the Hessian. The second partial derivatives of the relaxation are given by
\[
    \partial_{\theta_j, \theta_k} \E_{\theta, \sigma}(f) = \delta_{jk} \cdot \big(2  + a_j \xi_j^2 \cos(\xi_j \theta_j) \exp\big( -\tfrac{1}{2}\sigma^2 \xi_j^2\big)\big) \,,
\]
where $\delta_{jk}$ is the Kronecker delta.
Thus, the Hessian is diagonal with strictly positive values at $\theta \in \R^n$ if and only if for all $j \in \N_{\leq n}$
\[
    -2 < a_j \xi_j^2 \cos(\xi_j \theta_j) \exp\big( -\tfrac{1}{2}\sigma^2 \xi_j^2\big) \,.
\]
As $\cos{(\xi_j\theta_j)}\geq -1$, we know that this is the case at least if
\[
    2 > a_j \xi_j^2 \exp\big(-\tfrac{1}{2}\sigma^2 \xi_j^2\big) \iff \tfrac{2}{\xi_j^2}\log\big(\tfrac{a_j \xi_j^2}{2}\big) < \sigma^2.
\]
Picking ${\sigma^*}^2 > \max_{j \in \N_{\leq n}} \tfrac{2}{\xi_j^2}\log\big(\tfrac{a_j \xi_j^2}{2} \big)$, we obtain the result.
\end{example}

The above result outlines the very structure that can be exploited for optimization by first-order optimization methods \cite{bubeck2015convex}. In some practical settings, a gradient method on the mean parameter $\theta$ may be augmented by a manipulation of the standard deviation $\sigma$. However, keeping $\sigma$ large improves robustness.

What follows is the general setting used throughout the work. To this end, let $\mathcal{L}^1(\Omega, \mathcal{A}, \Ps{\theta})$ denote the at least once Lebesgue integrable, real-valued functions on the measure space $(\Omega, \mathcal{A}, \Ps{\theta})$.

\begin{definition}[Stochastic Relaxation]
\label{def:relProb}
Given
\begin{itemize}

    \item a family of probability measures $\, \{ \Ps{\theta} : \theta \in \Theta \}$ on a measure space $(\Omega, \mathcal{A})$,

    \item an optimization problem
    $ \,\, \displaystyle
        \min_{x \in \Omega} \,\, f(x) \,\,, \, \text{where} \,\, f: \Omega \to \R\
    $, and
    
    \item $f \in \mathcal{L}^1(\Omega, \mathcal{A}, \Ps{\theta})\,, \,\, \forall \theta \in \Theta$,
\end{itemize}
we call the problem
\[
    \min_{ \theta \in \Theta } \,\, \E_\theta (f)
        =
    \min_{ \theta \in \Theta } \,\, \int_{\Omega} f \, \mathrm{d} \Ps{\theta} \,,
\]
the \textbf{stochastic relaxation} of $f$ for which we write $(f, \Omega, \mathcal{A}, (\Ps{\theta})_{\theta \in \Theta})\,$. In case $\Omega$ has a metric $d$ such that $\mathcal{B}(d) = \mathcal{A}$ we write
$(f, d, \Omega, \mathcal{A}, (\Ps{\theta})_{\theta \in \Theta})$.\footnote{We denote the Borel $\sigma$-algebra of a topological space $\Omega$ by $\mathcal{B}(\Omega)$. If $\Omega$ admits a metric $d$, we also write $\mathcal{B}(d)$.}
\end{definition}

\begin{remark}
    In principle, there is no need for requiring the measures $(\Ps{\theta})_{\theta \in \Theta}$ to be non-negative. One could define a relaxed problem analogously by allowing signed measures. In this case however, the relation of the relaxed problem and the original problem in
    \cref{sec:Consistency in the Codomain} or the results of \cref{sec:Weak Convexity}
    will in general not hold.
\end{remark}

Inspired by the previous example, we will now aim to generalize the traits of \cref{ex:Rastrigin} to fit the picture of the practical use of stochastic relaxations.

\section{Consistency in the Codomain}
\label{sec:Consistency in the Codomain}

Manipulating an optimization problem raises the question of whether solutions or the cost of solutions will relate to those of the original optimization problem defined by a function $f: \Omega \to \R$.

We can observe that some distinguished generalized candidates can be associated with a set of close-to-optimal candidates of $f$ by their function value and mass distribution in candidate space. Formally, we define the following.

\begin{definition}[$f$-consistent]
    Let $(f, d, \Omega, \mathcal{A}, (\Ps{\theta})_{\theta \in \Theta})$ be a relaxation as defined in \cref{def:relProb} and let $f: \Omega \to \R$ have a unique global minimum at $x^*$. We call the relaxation $f$-\textbf{consistent} if
    \[
        \forall x \in \Omega\,,\,\, x \neq x^* : \exists \theta^* \in \Theta:
     \,\, \E_{\theta^*}(f) < f(x) \,.
    \]
\end{definition}

We will learn that requiring the Dirac measure at the global optimum of $f$ to be a limit candidate of the relaxation in conjunction with a bounding attribute of $f$ is sufficient for a relaxation to be $f$-consistent.
A concise way of treating the problem based on measure theory is presented in this section. The results cover a general setting that includes the candidate spaces $\Omega \in \{\mathbb{Z}^n, \R^n\}$.

\begin{definition}[$f$-$\varepsilon$-$\delta$-concentrated, $f$-concentrated]
\label{def:con}
Let $(f, d, \Omega, \mathcal{A}, (\Ps{\theta})_{\theta \in \Theta})$ be a relaxation as defined in \cref{def:relProb}. If $\varepsilon, \delta > 0\,$, the measure $\Ps{\theta}$ is called $f$-$\varepsilon$-$\delta$-\textbf{concentrated} at $x \in \Omega$ if
\[
   \int_{ \Omega - U_\delta(x) } \max \{\abs{f}, 1\} \, \mathrm{d}\Ps{\theta} < \varepsilon \,,
\tag*{(CON)}
\label{eq:defRelProbCon}
\]
where $U_\delta(x)$ denotes the $\delta$-ball centered at $x$.
Moreover, probability measures $(\Ps{\theta})_{\theta \in \Theta}$ are called $f$-\textbf{concentrated} at $x \in \Omega$ if 
\[
    \forall \varepsilon, \delta > 0: \exists \theta \in \Theta: \,\, \Ps{\theta} \,\, \text{is $f$-$\varepsilon$-$\delta$ concentrated at $x$}\,.
\]
\end{definition}

A general approximation result will pave the way toward approximation at a global optimum of $f$.

\begin{lemma}[Approximation]
\label{lem:approx}
Let $(f, d, \Omega, \mathcal{A}, (\Ps{\theta})_{\theta \in \Theta})$ be a relaxation as defined in \cref{def:relProb}, where $(\Ps{\theta})_{\theta \in \Theta}$ are $f$-concentrated at $x \in \Omega$ and $f$ is continuous at the same $x \in \Omega$. We have
\[
     \forall \gamma > 0 : \exists \theta \in \Theta:
     \,\, \abs{\E_{\theta}(f)-f(x)} < \gamma \,.
\]
\end{lemma}
\begin{proof}
    Let $\gamma > 0\,$. Due to continuity of $f$ at $x \in \Omega$, we pick $\delta > 0$ such that
    \begin{equation}
    \label{eq:lemApproxCont}
        d(x,y) < \delta 
        \implies
        \abs{f(x) - f(y)} < \tfrac{\gamma}{3} \,.
    \end{equation}
    
    As $(\Ps{\theta})_{\theta \in \Theta}$ are $f$-concentrated at $x$, we pick $\theta^*$ such that
    \begin{equation}
    \label{eq:lemApproxConcent}
        \Ps{\theta^*} \,\, \text{is $f$-$\varepsilon$-$\delta$-concentrated at $x$}\,,
    \end{equation}
    where $\varepsilon := \min \big\{ \tfrac{\gamma}{3}, \tfrac{\gamma}{3\abs{f(x)}} \big\}$.\\

    It follows, that
    \begingroup
    \begin{align*}
        \abs{\E_{\theta^*}(f)-f(x)}
        &=
        \bigg|
            \int_{\Omega - U_\delta(x)} f \, \mathrm{d}\Ps{\theta^*}
            + \int_{U_\delta(x)} f \, \mathrm{d}\Ps{\theta^*}
            - f(x) 
        \bigg|
        \tag*{\small{\text{(additivity}})}
        \\
        &\leq
        \int_{\Omega - U_\delta(x)} \abs{f} \, \mathrm{d}\Ps{\theta^*}
        +
        \bigg|
            \int_{U_\delta(x)} f \, \mathrm{d}\Ps{\theta^*}
            - f(x)
        \bigg|
        \tag*{\small{\text{(triangle inequality}})}
        \\
        &<
        \frac{\gamma}{3}
        +
        \bigg|
            \int_{U_\delta(x)} f \, \mathrm{d}\Ps{\theta^*}
            - f(x)
        \bigg|
        \tag*{\small{\text{(\cref{eq:lemApproxConcent}}})}
        \\
        &=
        \frac{\gamma}{3}
        +
        \bigg|
            \int_{U_\delta(x)} f - f(x)\, \mathrm{d}\Ps{\theta^*}
            +
             f(x)\Ps{\theta^*}\big( \Omega - U_\delta(x) \big)
        \bigg|
        \tag*{\small{\text{(additivity}})}
        \\
        &\leq
        \frac{\gamma}{3}
        +
        \int_{U_\delta(x)} \abs{f - f(x)} \, \mathrm{d}\Ps{\theta^*}
        +
         \abs{f(x)} \Ps{\theta^*}\big( \Omega - U_\delta(x) \big)
        \tag*{\small{\text{(triangle inequality}})}
        \\
        &\leq
        \frac{\gamma}{3}
        +
        \frac{\gamma}{3} \Ps{\theta^*}\big( U_\delta(x) \big)
        +
        \abs{f(x)} \Ps{\theta^*}\big( \Omega - U_\delta(x) \big)
        \tag*{\small{\text{(\cref{eq:lemApproxCont}}})}
        \\
        &<
        \frac{\gamma}{3}
        +
        \frac{\gamma}{3}
        +
        \frac{\gamma}{3}
        = \gamma
        \tag*{\small{\text{(probability measure and \cref{eq:lemApproxConcent}}})}
    \end{align*}
    \endgroup
\end{proof}

\begin{remark}
    In general, we can not relax the requirement of $f$-concentration as specified by \ref{eq:defRelProbCon} and still fulfill the approximation property given in \cref{lem:approx}. Consider the following examples.
    \begin{enumerate}
        \item Let $f: \R \to \R, x\mapsto x^2$, $\delta_x$ denote the Dirac measure at $x \in \R$, and
        \[
            (\Omega, \mathcal{A}, \Ps{\theta}) \overset{!}{=} \big(\R, \mathcal{B}(\R),  (1-\theta) \delta_0 + \theta \delta_{1/\theta}\big) \,, \,\, \forall \theta \in (0,1) \, .
        \]
        Clearly, we have a concentration of measure at $x=0$, as
        \[
            \Ps{\theta}( \Omega - \{ 0 \}) = \theta \,,\,\, \forall \theta \in (0,1)\,.
        \]
        However, we can not approximate $f$ at $0$ as in \cref{lem:approx}, since for all $\theta \in (0,1)$, we have
        \[
            \E_\theta(f) = \tfrac{1}{\theta} > 1 > 0 = f(0) \,.
        \]
        
        \item Let $f: \R \to \R, x \mapsto -\exp\big(-x^2\big)$ and 
        \[
            (\Omega, \mathcal{A}, \Ps{\theta}) \overset{!}{=} (\R, \mathcal{B}(\R), \delta_\theta) \,, \,\, \forall \theta \in (1, \infty) \, .
        \]
        Clearly, the $L^1$-norm of our function outside of $0$ can be arbitrarily small, i.e.,
        \[
            \forall \varepsilon > 0 : \exists \theta \in (1, \infty) : \int_{\Omega - \{ 0 \}} \abs{f} \, \mathrm{d}\delta_\theta \leq \exp\big(-\theta^2\big) < \varepsilon \,.
        \]
        However, we can not approximate $f$ at $0$ as in \cref{lem:approx}, since for all $\theta \in (1, \infty)$, we have
        \[
            \E_\theta(f) = -\exp\big(-\theta^2\big) > - \tfrac{1}{e} > -1 = f(0) \,.
        \]
    \end{enumerate}
\end{remark}

\begin{theorem}[Consistency]
\label{thm:consistent}
    Let $(f, d, \Omega, \mathcal{A}, (\Ps{\theta})_{\theta \in \Theta})$ be a relaxation as defined in \cref{def:relProb} and let $f: \Omega \to \R$ be continuous at its unique global minimum at $x^*$. Furthermore, assume that $(\Ps{\theta})_{\theta \in \Theta}$ are $f$-concentrated at $x^*$. Then, the relaxation is $f$-consistent.
\end{theorem}
\begin{proof}
    Pick $x \in \Omega$ such that $x \neq x^*$. Define
    \begin{equation}
        \gamma := f(x) - f(x^*) \,,
    \label{eq:thmPresOpt}
    \end{equation}
    which is strictly positive as $x^*$ is the unique global minimum of $f$.
    
    By assumption, $(\Ps{\theta})_{\theta \in \Theta}$ are $f$-concentrated at $x^*$. Thus, by \cref{lem:approx} we find $\theta^* \in \Omega$ such that
    \begin{equation}
        \abs{ \E_{\theta^*}(f) - f(x^*) } < \gamma \,.
    \label{eq:thmPresConc}
    \end{equation}
    Applying \cref{eq:thmPresOpt} in the first step, \cref{eq:thmPresConc} in the second step and the fact that $x^*$ is a global minimum of $f$ in the third, we obtain the result
    \begin{align*}
        f(x)
        &=
        \gamma + f(x^*)
        \\
        &>
        \abs{ \E_{\theta^*}(f) - f(x^*) } + f(x^*)
        \\
        &=
        \E_{\theta^*}(f) - f(x^*) + f(x^*)
        \\
        &=
        \E_{\theta^*}(f) \,.
    \end{align*}
\end{proof}

A consistent relaxation ensures that minimizing the relaxation results in function values arbitrarily close to those of the $f$ at its optimum $x^* \in \Omega$. Therefore, we find $\Ps{\theta}$ with an arbitrarily close-to-one probability of an arbitrarily low regret. In the limiting case, we can also observe the following immediate result.

\begin{corollary}
    In the setting of \cref{thm:consistent}, if the relaxation has a minimum, it is unique at the Dirac measure of the minimum of $f$, that is, at $\delta_{x^*}$.
\end{corollary}

\section{Imposing Differentiability}
\label{sec:Imposing Differentiability}

In this section, we derive gradient representations of relaxations. The results in \cref{theorem:diff,cor:diff} cover cases where the objective function transformed by the relaxation may not be differentiable or no differentiable structure of the candidate space $\Omega$ is specified.\\

First, however, let us look at the simple case where the candidate space $\Omega=\Theta=\R^n$, the objective function $f: \Omega \to \R$ is continuously differentiable, and there is a distinguished probability measure $\Ps{}$ on $(\Omega, \mathcal{B}(\Omega))$. We can construct a family of probability measures by setting
\[
    \Ps{\theta}(A) := \Ps{}(A-\theta)
\]
for all $\theta \in \Theta$ and for all $A \in \mathcal{A} := \mathcal{B}(\Omega)$.
This case is critical in practice and particularly interesting due to the following property.
\begin{theorem}[Preserving Differentiability]
    Let a relaxation as defined in \cref{def:relProb} be given by $(f, \R^n, \mathcal{B}(\R^n), (\Ps{\theta})_{\theta \in \R^n})$ with
    \begin{itemize}
    \item $\Ps{\theta}(A) = \Ps{0}(A-\theta)$ for all $\theta \in \R^n$ and for all $A \in \mathcal{B}(\R^n)$
    
    \item $f: \R^n \to \R$ is continuously differentiable, and
    
    \item there exists $\varepsilon > 0$ such that $x \in \R^n \mapsto \max \abs{(\partial_r f)(U_\varepsilon(x))}$ has a $\Ps{0}$-integrable upper bound for all $r \in \N_{\leq n}$.
    \end{itemize}
    We have $\nabla_\theta \E_\cdot(f) \equiv \E_\cdot(\nabla_x f)$. If $\nabla_x f$ is Lipschitz/uniformly continuous, then $\nabla_\theta \E_\cdot(f)$ is as well.
\end{theorem}
\begin{remark}
    The third assumption deviates from the standard assumption of Leibniz's integral rule, as we have to restrict the parameter set to some neighborhood around any parameter to obtain a bound that fits Leibniz's integral rule. As differentiability is a local property, this still suffices for the statement of the theorem.
\end{remark}
\begin{proof}
Let $\psi(\theta, x) := \psi_\theta(x) := x+\theta$ for all $x \in \R^n$ and for all $\theta \in \R^n$, then
    \begin{align*}
    \int_\Omega f \, \mathrm{d}\Ps{\theta}
    &=
    \int_0^\infty \Ps{\theta}(\{ x \in \R^n \mid f(x) > t \}) \, \mathrm{d}t
    \tag*{\small{\text{(definition Lebesgue integral)}}}
    \\
    &=
    \int_0^\infty \Ps{0}\big(\psi_\theta^{-1}(\{ x \in \R^n \mid f(x) > t \})\big) \, \mathrm{d}t
    \tag*{\small{\text{(assumption)}}}
    \\
    &=
    \int_0^\infty \Ps{0}(\{ x \in \R^n \mid (f\circ \psi_\theta)(x) > t \}) \, \mathrm{d}t
    \tag*{\small{\text{(definition inverse)}}}
    \\
    &=
    \int_\Omega f\circ\psi_\theta \, \mathrm{d}\Ps{0} \,.
    \tag*{\small{\text{(definition Lebesgue integral)}}}
    \end{align*}
    Due to the last assumption, we can apply Leibniz's integral rule \cite[p.~91,~Theorem~11.5]{schilling2005}, and get
    \begin{align*}
        \nabla_\theta \int_\Omega f \circ \psi_\theta \, \mathrm{d}\Ps{0}
        &=
        \int_\Omega (\nabla_\theta f \circ \psi)(\theta, x) \, \Ps{0}(\mathrm{d} x)
    \tag*{\small{\text{(Leibniz integral rule)}}}
        \\
        &=
        \int_\Omega (\nabla_x f)(x+\theta) \, \Ps{0}(\mathrm{d} x)
    \tag*{\small{\text{(definition $\psi$)}}}
        \\
        &=
        \int_\Omega (\nabla_x f) \, \mathrm{d}\Ps{\theta} \,.
    \tag*{\small{\text{(above equation)}}}
    \end{align*}
    Combining the results, we obtain for all $\theta \in \R^n$, that
    \[
        \nabla_\theta \E_\theta(f) = \E_\theta(\nabla_x f)\,.
    \]
    In case the gradient of $f$ is also Lipschitz with constant $L$, we have
    \begin{align*}
        \norm{\nabla_\theta \E_{\theta'}(f) -  \nabla_\theta \E_\theta(f)}
        &\leq
        \int_\Omega \norm{(\nabla_x f)(x+\theta')-(\nabla_x f)(x+\theta)} \, \Ps{0}(\mathrm{d} x)
    \tag*{\small{\text{(above derivation)}}}
        \\
        &\leq
        \int_\Omega L \norm{\theta'-\theta} \, \Ps{0}(\mathrm{d} x)
    \tag*{\small{\text{(Lipschitzness of $f$)}}}
        \\
        &= L \norm{\theta'-\theta} \,.
        \tag*{\small{\text{(Lebesgue integral \& probability measure)}}}
    \end{align*}
    By an analogous derivation, we obtain continuity of $\nabla_\theta \E_\cdot(f)$ if $\nabla_x f$ is continuous.
\end{proof}

We will now look again at the abstract, general setting for relaxations. A transfer of smoothness from the probability measures $\{ \Ps{\theta} \mid \theta \in \Theta \}$ to $\E_\cdot(f)$ is also central to applications. Tying together many notions of differentiability, e.g., for parameter-free distributions, we consider pullbacks using $\gamma: U \subseteq \R^n \to \Theta$ to describe the transfer of (Lipschitz) differentiablity.

\begin{theorem}
\label{theorem:diff}
    Let $(f, \Omega, \mathcal{A}, (\Ps{\theta})_{\theta \in \Theta})$ be a relaxation as defined in \cref{def:relProb} and let $\gamma: U \to \Theta$ be a function for some open $U \subseteq \R^n$.
    
    \setlist[description]{font=\normalfont\itshape}
    \begin{description}
    \item[(Partial Differentiability)] If
    \begin{itemize}
        \item there exists a measure $\mu$ on $(\Omega, \mathcal{A})$ s.t. for all $t \in U$ there exists a density $k_{\gamma(t)}$ of $\Ps{\gamma(t)}$ w.r.t. $\mu$,
        \item for all $x \in \Omega$ the partial derivatives of $t \in U \mapsto k_{\gamma(t)}(x)$ exist, and
        \item for all $r \in \N_{\leq n}$ there exists $g_r \in \mathcal{L}^1(\Omega, \mathcal{A}, \mu)$ s.t. for all $(t, x) \in U \times \Omega$ we have 
        \[
            \abs{f(x) \partial_r k_{\gamma(t)}(x)} \leq g_r(x) \,,\footnote{The operator $\partial_r$ refers to the partial derivative with respect to the $r$-th argument of the function $\gamma$.}
        \]
    \end{itemize}
    then the partial derivatives of $t \in U \mapsto \E_{\gamma(t)}(f)$ exist and
    \[
        \partial_r \E_{\gamma(t)}(f) = \int_\Omega f \partial_r k_{\gamma(t)} \, \mathrm{d}\mu \,.
    \]
    
    \item[(Continuous Differentiability)] If in addtion, for all $r \in \N_{\leq n}$ and for all $x \in \Omega$, $t \in U \mapsto \partial_r k_{\gamma(t)}(x)$ is continuous, then $t \in U \mapsto \E_{\gamma(t)}(f)$ is continuously differentiable.
    
    \item[(Lipschitz Differentiability)] If in addition, for all $r \in \N_{\leq n}$ and for all $x \in \Omega$, there exists $L_r(x) \in \R$ s.t. for all $t',t \in U$
    \begin{align*}
        \abs{\partial_r k_{\gamma(t')}(x) - \partial_r k_{\gamma(t)}(x)}
        \leq L_r(x) \norm{t' - t}
        \quad \text{and} \quad f L_r \in \mathcal{L}^1(\Omega, \mathcal{A}, \mu) \,,
    \end{align*}
    then $t \in U \mapsto \nabla \E_{\gamma(t)}(f)$ is Lipschitz with constant
    \[
        \sum_{r=1}^n \int_\Omega \abs{f} L_r \, \mathrm{d}\mu \, .
    \]
\end{description}
\end{theorem}
\begin{samepage}
\begin{proof}\leavevmode
\setlist[description]{font=\normalfont}
\begin{description}
    \item[(Partial Differentiability)] To obtain the first result, we apply Leibniz's integral rule \cite[p.~91,~Theorem~11.5]{schilling2005}. The requirements are fulfilled as per the following arguments.
    \begin{enumerate}[label=(\alph*)]
        \item For all $t \in U$, we have
        \vspace{-3mm}
        \[
            x \in \Omega \mapsto f(x)k_{\gamma(t)}(x) \in \mathcal{L}^1(\Omega, \mathcal{A}, \mu)
        \]
        as $f \in \mathcal{L}^1(\Omega, \mathcal{A}, \Ps{\gamma(t)})$ and $k_{\gamma(t)}\mu = \Ps{\gamma(t)}$ $\mu$-almost everywhere.
        
        \item For all $x \in \Omega$, the function $t \in U \mapsto f(x)k_{\gamma(t)}(x)$ is partially differentiable as $t \in U \mapsto k_{\gamma(t)}(x)$ is partially differentiable.
        
        \item In each coordinate this requirement is directly fulfilled by our initial assumptions.
    \end{enumerate}
    As $U$ is open in $\R^d$, we can apply Leibniz's integral rule to each coordinate direction.
    
    \item[(Continuous Differentiability)] To obtain the second result, we apply the Lebesgue dominated convergence theorem, e.g., see \cite[p.~89]{schilling2005}, to an arbitrary convergent sequence $(t_j)_{j\in \N} \to t \in U$. Clearly, for all $r \in \N_{\leq n}$ and for all $x \in \Omega$
    \vspace{-1mm}
    \begin{align*}
        \partial_r k_{\gamma(t_j)}(x) \xrightarrow{j \to \infty} \partial_r k_{\gamma(t)}(x) \quad
        \implies \quad
        f(x)\partial_r k_{\gamma(t_j)}(x) \xrightarrow{j \to \infty} f(x)\partial_r k_{\gamma(t)}(x) & \,.
    \end{align*}
    The left side of the above statement is fulfilled by our additional continuity assumption. The boundedness requirement is fulfilled by the assumption of the first statement. Therefore, the Lebesgue dominated convergence theorem tells us for all $r \in \N_{\leq n}$
    \vspace{-1mm}
    \[
        \lim_{j \to \infty} \int_\Omega f \partial_r k_{\gamma(t_j)} \, \mathrm{d}\mu = \int_\Omega f \partial_r k_{\gamma(t)} \, \mathrm{d}\mu \,.
    \]
    Combined with the first result we get that for all $r \in \N_{\leq n}$
    \[
         \lim_{j \to \infty} \partial_r \E_{\gamma(t_j)}(f) = \partial_r \E_{\gamma(t)}(f) \,,
    \]
    i.e., continuity of the partial differentials. This implies continuity of the differential.
    
    \item[(Lipschitz Differentiability)] To obtain the third result, we observe that for all $r \in \N_{\leq n}$, for all $x \in \Omega$, and for all $t',t \in U$
    \vspace{-2mm}
    \begin{align*}
        &&
        \abs{\partial_r k_{\gamma(t')}(x) - \partial_r k_{\gamma(t)}(x)}
        \leq L_r(x) \norm{t' - t} &
        \\
        \implies &&
        \abs{f(x)\partial_r k_{\gamma(t')}(x) - f(x)\partial_r k_{\gamma(t)}(x)}
        \leq \abs{f(x)}L_r(x) \norm{t' - t} & \,.
    \end{align*}
    The left side of the above statement is fulfilled by our additional Lipschitz assumption. Using the result of the first statement and the triangle inequality this implies that for all $t', t \in U$
    \begin{align*}
        \norm{ \nabla\E_{\gamma(t')}(f) - \nabla\E_{\gamma(t)}(f) } 
        &=
        \norm{\int_\Omega  f \nabla k_{\gamma(t')} - f \nabla k_{\gamma(t)} \, \mathrm{d}\mu}
        \\
        &\leq
        \int_\Omega \norm{f\nabla k_{\gamma(t')} -  f\nabla k_{\gamma(t)}} \, \mathrm{d}\mu
        \\
        &\leq
        \int_\Omega \sum_{r=1}^n \abs{f \partial_r k_{\gamma(t')} - f \partial_r k_{\gamma(t)}} \, \mathrm{d}\mu
        \\
        &\leq
        \int_\Omega \sum_{r=1}^n \abs{f} L_r \norm{t' - t} \, \mathrm{d}\mu
        \\
        &=
        \bigg(\sum_{r=1}^n \int_\Omega \abs{f} L_r \, \mathrm{d}\mu \bigg)\norm{t' - t} \,.
    \end{align*}
\end{description}
\end{proof}
\end{samepage}
\begin{remark}
    In \cref{theorem:diff}, we may relax the requirements to hold for all $x \in \Omega$ $\mu$-almost everywhere only.
\end{remark}

From \cref{theorem:diff}, we can derive the following well-known result, which is useful for gradient-based optimization methods.

\begin{corollary}
\label{cor:diff}
    Let $(f, \Omega, \mathcal{A}, (\Ps{\theta})_{\theta \in \Theta})$ be a relaxation as defined in \cref{def:relProb}, with the properties
    \begin{itemize}
        \item $\Theta$ is an open subset of $\R^n$,
        \item the measures have $\theta$-Lipschitz differentiable densities with respect some measure $\mu$, i.e. $\Ps{\theta} = k_\theta \mu$ for all $\theta \in \Theta$, and
        \item the regularity assumptions of \cref{theorem:diff} with $\gamma \equiv \mathrm{id}_{\Theta}$ are fulfilled.\footnote{$\mathrm{id}_{\Theta}$ denotes the identity map of $\Theta$.}
    \end{itemize}
    Then, $\nabla_\theta \E_\cdot(f) \equiv  \E_\mu(f \nabla_\theta k_\cdot) \equiv \E_\cdot(f \nabla_\theta \ln k_\cdot)$ is Lipschitz.
\end{corollary}
\begin{proof}
This result is an application of the "log-likelihood trick".
    \begin{align*}
        \nabla_\theta \E_\cdot(f)
        &\equiv
        \int_{\Omega} f(x) \nabla_\theta k_\cdot(x) \, \mu(\mathrm{d}x)
    \tag*{\small{\text{(by \cref{theorem:diff})}}}
        \\
        &\equiv
        \int_{\mathrm{supp}\, k_\cdot} f(x) (\nabla_\theta k_\cdot)(x) \tfrac{k_\cdot(x)}{k_\cdot(x)} \, \mu(\mathrm{d}x)
    \tag*{\small{\text{($\tfrac{k_\cdot(x)}{k_\cdot(x)} \equiv 1$)}}}
        \\
        &\equiv
        \int_{\mathrm{supp}\, k_\cdot} f(x) (\nabla_\theta \ln k_\cdot)(x) k_\cdot(x) \, \mu(\mathrm{d}x)
    \tag*{\small{\text{(log-likelihood trick)}}}
        \\
        &\equiv
        \int_{\mathrm{supp}\, k_\cdot} f(x) (\nabla_\theta \ln k_\cdot)(x) \, \Ps{\cdot}(\mathrm{d}x)
    \tag*{\small{\text{(by definition $k_\cdot \mu \equiv \Ps{\cdot}$)}}}
        \\
        &\equiv
        \E_\cdot(f \nabla_\theta \ln k_\cdot)
    \tag*{\small{\text{(definition expected value, $\ln 0 := 0$ on null sets)}}}
    \end{align*}
\end{proof}

Especially in black-box, noisy, or unstructured problems $f$ the optimization of $\E_\cdot(f)$ using gradient-based methods is a reasonable alternative to finite-difference methods. The transfer of Lipschitzness of gradients can be exploited for improved convergence guarantees \cite{bubeck2015convex}. In the literature, the properties derived in this section have been observed experimentally for many algorithms employing stochastic relaxations at their core, specifically in the setting of \cref{ex:Rastrigin}.

\section{Weak Convexity}
\label{sec:Weak Convexity}

The notion of convexity of a function is central to continuous optimization. Non-convex optimization problems are much harder to solve in general. However, for families of translated probability measures---including multivariate normal distributions---we show that a function's convexity is preserved in its relaxation. More importantly, we provide conditions under which non-convex functions result in convex stochastic relaxations. To this end, consider the following class of functions $f$ in case of $\Omega=\Theta=\R^n$.

\begin{definition}[Weak Convexity]
\label{def:Weak Convexity}
    Let a relaxation as defined in \cref{def:relProb} be given by $(f, \R^n, \mathcal{B}(\R^n), (\Ps{\theta})_{\theta \in \R^n})$ with
    \[
        \Ps{\theta}(A) = \Ps{0}(A-\theta) \,\, \text{for all $\theta \in \R^n$ and for all $A \in \mathcal{B}(\R^n)$.}
    \]
    We call $f$ (strictly/ m-strongly) $\Ps{0}$-\textbf{weakly convex} if there is $\sigma^*>0$ such that for all $\sigma > \sigma^*$
    \[
        \Ps{\theta}^\sigma(A) := \Ps{0}(\sigma^{-1}(A-\theta)) \,\, \text{for all $\theta \in \R^n$ and for all $A \in \mathcal{B}(\R^n)$}
    \]
    gives a (strictly/ m-strongly) convex relaxation $\theta \in \R^n \mapsto \E_{\theta, \sigma}(f)$.
\end{definition}
As a direct result, we obtain the following by substitution.
\begin{lemma}
    In the setting of \cref{def:Weak Convexity}, if the probability measure $\Ps{0}$ has a density $k_0$ with respect to the Lebesgue measure $\lambda$, then for all $\sigma > 0$, for all $\theta \in \R^n$, and for all $A \in \mathcal{B}(\R^n)$, we have
    \begin{align*}
        \Ps{\theta}^\sigma(A)
        &=
        \Ps{0}(\sigma^{-1}(A-\theta))
    \tag*{\small{\text{(definition $\Ps{\theta}^\sigma$)}}}
        \\
        &=
        \int_{\sigma^{-1}(A-\theta)} k_0 \, \mathrm{d}\lambda
    \tag*{\small{\text{(definition density)}}}
        \\
        &=
        \int_{A} \sigma^{-n} k_0(\sigma^{-1}(\cdot -\theta)) \, \mathrm{d}\lambda \,.
    \tag*{\small{\text{(substitution)}}}
    \end{align*}
    That is, $k_{\theta, \sigma} := \sigma^{-n} k_0(\sigma^{-1}(\cdot -\theta))$ is a Lebesgue density of $\Ps{\theta}^\sigma$.
\end{lemma}
The goal for the rest of the section will be the derivation of a weak convexity property for stochastic relaxations. The approach we take will be based on Fourier analysis and will work with a relaxation of the superposition $f = r + g: \R^n \to \R$ based on a class of densities that are rapidly decreasing. The function $r$ will be assumed to be strongly convex, and $g$ can be understood as a disturbance. While in \cref{ex:Rastrigin}, $g$ was a finite sum of cosines and the densities were isotropic Gaussian, we aim to find a general setting that provides interpretation for the success of stochastic relaxations in global and robust optimization.

\begin{theorem}[Preserving Convexity]
\label{thm:strongconv}
    Let a relaxation as defined in \cref{def:relProb} be given by $(f, \R^n, \mathcal{B}(\R^n), (\Ps{\theta})_{\theta \in \R^n})$ with
    \begin{itemize}
        \item  $\Ps{\theta}(A) = \Ps{0}(A-\theta)$ for all $\theta \in \R^n$ and for all $A \in \mathcal{B}(\R^n)$, and
        \item $f$ is strongly convex, i.e., there exists $m > 0$ such that for all $v,w \in \R^n$ and for all $t \in [0, 1]$
        \[
            f(t v + (1-t) w) \leq t f(v) + (1-t) f(w) - \tfrac{1}{2} m t(1-t) \norm{v-w}^2 \,.
        \]
    \end{itemize}
    Then, the relaxation $\theta \in \R^n \mapsto \E_\theta(f)$ is $m$-strongly convex.
\end{theorem}
\begin{proof}
    Let $\psi(\theta, x) := \psi_\theta(x) := x+\theta$ for all $x \in \R^n$ and for all $\theta \in \R^n$. Further let $v,w \in \R^n$ and $t \in [0,1]$, then
\begingroup
\allowdisplaybreaks
    \begin{align*}
        \E_{tv + (1-t)w}(f)
        &=
        \int_{\R^n} f \, \mathrm{d}\Ps{tv + (1-t)w}
    \tag*{\small{\text{(definition expected value)}}}
        \\
        &=
        \int_0^\infty \Ps{tv + (1-t)w}(\{ x \in \R^n \mid f(x) > t \}) \, \mathrm{d}t
    \tag*{\small{\text{(definition Lebesgue integral)}}}
        \\
        &=
        \int_0^\infty \Ps{0}\big(\psi_{tv + (1-t)w}^{-1}(\{ x \in \R^n \mid f(x) > t \})\big) \, \mathrm{d}t
    \tag*{\small{\text{(assumption)}}}
        \\
        &=
        \int_0^\infty \Ps{0}(\{ x \in \R^n \mid (f\circ \psi_{tv + (1-t)w})(x) > t \}) \, \mathrm{d}t
    \tag*{\small{\text{(definition inverse)}}}
        \\
        &=
        \int_{\R^n} f(x+ tv + (1-t)w) \, \Ps{0}(\mathrm{d}x)
    \tag*{\small{\text{(definition Lebesgue integral)}}}
        \\
        &=
        \int_{\R^n} f(t(x+v) + (1-t)(x+w)) \, \Ps{0}(\mathrm{d}x) 
    \tag*{\small{\text{($x = t x + (1-t) x$)}}}
        \\
        &\leq
        \int_{\R^n} t f(x+v) + (1-t) f(x+w)
        - \tfrac{1}{2} m t(1-t) \norm{v-w}^2 \, \Ps{0}(\mathrm{d}x) 
    \tag*{\small{\text{(third assumption)}}}
        \\
        &=
        t \int_{\R^n} f(x+v) \, \Ps{0}(\mathrm{d}x)
        +
        (1-t) \int_{\R^n} f(x+w) \, \Ps{0}(\mathrm{d}x) 
        - \tfrac{1}{2} m t(1-t) \norm{v-w}^2 
    \tag*{\small{\text{(linearity, probability measure)}}}
        \\
        &=
        t \E_v(f)
        +
        (1-t) \E_w(f)
        - \tfrac{1}{2} m t(1-t) \norm{v-w}^2 \,.
    \tag*{\small{\text{(first steps reversed)}}} 
    \end{align*}
\endgroup
\end{proof}
\begin{remark}
    Preservation also holds for (strictly) convex functions.
\end{remark}

The following result is developed in preparation of \cref{thm:fourier}.
To retain the generality of this section and obtain a convexity result similar to \cref{ex:Rastrigin}, the distributional formulation of the Fourier transform is needed. This result will pave the way for very practical corollaries at the end of this section.

\begin{lemma}
\label{lem:fourier1}
    Let $\lambda$ be the Lebesgue measure
    \begin{itemize}
        \item $g: \R^n \to \R$ measurable with $T_g \in S'(\R^n)$,\footnote{$T_g$ denotes the functional $\varphi \mapsto \int_{\R^n} g \varphi \, \mathrm{d}\lambda$. This property ensures that the integral exists. Continuity and linearity is generally fulfilled. $\mathcal{S}$ denotes the Schwartz space and $\mathcal{S}'$ its dual space.}
        \item $\mathcal{F}(T_g)(\varphi) = 0$ for all $\varphi \in C_c^\infty(U_\varepsilon(0))$ for some $\varepsilon > 0$, and
    \end{itemize}
    Given $h \in \mathcal{S}(\R^n)$ and $\sigma \in \R$, then for all $y \in \R^n$ uniformly
    \[
        T_g\big(\sigma^{-n} h(\sigma^{-1} (\cdot-y))\big) = \int_{\R^n} g(x) \sigma^{-n} h(\sigma^{-1} (x-y)) \, \lambda(\mathrm{d}x) \xrightarrow{\sigma \to \infty} 0 \, .
    \]
    \end{lemma}
    \begin{proof}
    Applying general properties of the Fourier transform, we get the identity
    \begin{align*}
       T_g\big(\sigma^{-n} h(\sigma^{-1} (\cdot-y))\big)
       &=
       (\mathcal{F}^{-1}(\mathcal{F} T_g))\big(\sigma^{-n} h(\sigma^{-1} (\cdot-y))\big)
    \tag*{\small{\text{(Fourier inversion for distributions)}}}
       \\
       &=
       (\mathcal{F} T_g)\big(\mathcal{F}^{-1}\sigma^{-n} h(\sigma^{-1} (\cdot-y))\big)
    \tag*{\small{\text{(definition inverse transform)}}}
       \\
       &=
       (\mathcal{F} T_g)\big(e^{ i \langle y, \cdot \rangle} \sigma^{-n} \mathcal{F}^{-1}(h(\sigma^{-1} \cdot))\big)
    \tag*{\small{\text{(shift and linearity)}}}
       \\
       &=
       (\mathcal{F} T_g)\big(e^{ i \langle y, \cdot \rangle} \mathcal{F}^{-1}(h)(\sigma\cdot)\big) \,.
    \tag*{\small{\text{(scaling)}}}
    \end{align*}
    The Fourier transform is an automorphism on Schwartz space, therefore $s := \mathcal{F}^{-1}(h)$ is a Schwartz function.
    Now, due to continuity of $\mathcal{F} T_g$, we would like to have
    \[
        e^{i \langle y, \cdot \rangle} s(\sigma\cdot) \xrightarrow{\sigma \to \infty} 0 \,\, \text{for all $y \in \R^n$ uniformly in $\mathcal{S}$.}
    \]
    This, however, does not hold in general. We need the assumption ${\mathcal{F}T_g}_{|U_\varepsilon(0)} \equiv 0$.\\ To this end, consider a smooth partition of unity $u_1, u_2: \R^n \to \R$ with $u_1 + u_2 \equiv 1$ such that $\mathrm{supp} \, u_1 \subseteq  \overline{U_\varepsilon(0)}$ and $\mathrm{supp}\, u_2 \subseteq \R^n \setminus U_{\varepsilon/2}(0)$
    and see that
    \begin{align*}
        \mathcal{F}T_g\big( e^{ i \langle y, \cdot \rangle} s(\sigma\cdot)\big)
        &\equiv
        \mathcal{F}T_g\big((u_1+u_2) e^{ i \langle y, \cdot \rangle} s(\sigma\cdot)\big)
    \tag*{\small{\text{(definition $u_1$, $u_2$)}}}
        \\
        &\equiv
        \mathcal{F}T_g\big(u_1 e^{ i \langle y, \cdot \rangle} s(\sigma\cdot)\big) + \mathcal{F}T_g\big(u_2 e^{ i \langle y, \cdot \rangle} s(\sigma\cdot)\big)
    \tag*{\small{\text{(linearity)}}}
        \\
        &\equiv
        0 + \mathcal{F}T_g\big(u_2 e^{ i \langle y, \cdot \rangle} s(\sigma\cdot)\big) \,.
    \tag*{\small{\text{(assumption)}}}
    \end{align*}
    Due to $s$ being a Schwartz function and $\big|e^{i \langle y, \cdot \rangle}\big| \equiv 1$ for all $y \in \R^n$, we have
    \[
        u_2 e^{i \langle y, \cdot \rangle} s(\sigma\cdot) \xrightarrow{\sigma \to \infty} 0 \,\, \text{for all $y \in \R^n$ uniformly in $\mathcal{S}$.}
    \]
    By continuity of $\mathcal{F} T_g$ this implies the statement of the theorem.
    \end{proof}

Now, we will describe a setting for which the relaxed problem is convex when the original cost function may not be. As described previously, we consider the case where the function $f$ is a strongly convex function superimposed with a disturbance, impairing convexity and thus global optimization.

\begin{theorem}[Filtering]
\label{thm:fourier}
    Let $\lambda$ be the Lebesgue measure and $f: \R^n \to \R$ be a function that admits a decomposition
    \[
        f \equiv r + g
    \]
    such that
    \begin{itemize}
        \item $r$ is $m$-strongly convex,
        \item $\mathcal{F}(g)_{|U_\varepsilon(0)} \equiv 0$ in the distributional sense for some $\varepsilon > 0$, and
        \item $rs,gs \in \mathcal{L}^1(\R^n)$ for all Schwartz functions $s \in \mathcal{S}(\R^n)$.
    \end{itemize}
    Then for every non-negative, somewhere non-zero Schwartz function $k \in \mathcal{S}(\R^n)$ with $c_k := \big(\int_{\R^n} k \,\mathrm{d}\lambda\big)^{-1}$ and $m^* < m$, $f$ is $m^*$-strongly $(c_k k)\lambda$-weakly convex, i.e., there exists $\sigma_k^*$ such that for all $\sigma > \sigma_k^*$
    \[
        \theta \in \R^n \longmapsto \E_{\theta, \sigma}(f)
        = \int_{\R^n} f(x)c_k\sigma^{-n}k(\sigma^{-1}(x-\theta)) \, \lambda(\mathrm{d}x)
    \]
    is $m^*$-strongly convex.
\end{theorem}
\begin{proof}
    Pick a non-negative, somewhere non-zero Schwartz function $k \in \mathcal{S}(\R^n)$. For all $\sigma > 0$ and for all $\theta \in \R^n$, we have
    \[  
        \E_{\theta, \sigma}(f) = \E_{\theta, \sigma}(r) + \E_{\theta, \sigma}(g)
    \]
    and we know by direct application of \cref{thm:strongconv} that $\theta \in \R^n \mapsto \E_{\theta, \sigma}(r)$ is strongly convex with the same constant as $r$ for all $\sigma > 0$. We prove the result by showing that for all $\varepsilon > 0$ there exists $\sigma_\varepsilon$ such that all partial second derivatives of $\E_{\theta, \sigma_\varepsilon}(g)$ are uniformly bounded by $\varepsilon$. To see why this is, observe the following arguments.

    \begin{itemize}
        \item For all $i,j\in \N_{\leq n}$, $\sigma >0$, $\theta \in \R^n$, and $x \in \R^n$ there is the identity
        \[
            \partial_{\theta_i,\theta_j} k(\sigma^{-1}(x-\theta)) = \sigma^{-2}(\partial_{ij} k)(\sigma^{-1}(x-\theta))\,.
        \]
        \item By \cref{theorem:diff}, we have
        \begin{align*}
            \partial_{\theta_i,\theta_j} \E_{\theta, \sigma}(g) 
            = \sigma^{-2} \int_{\R^d} g(x) \sigma^{-n} c_k(\partial_{ij} k)(\sigma^{-1}(x-\theta)) \, \lambda(\mathrm{d}x) \,.
        \end{align*}
        \item The function $x \in \R^n \mapsto c_k (\partial_{ij}k)(x)$ is again a Schwartz function and together with $g$ fulfills the assumptions of \cref{lem:fourier1} for all $i,j \in \N_{\leq n}$.
        This provides a uniform bound of $\varepsilon$ for all $\sigma > \sigma_{\varepsilon, i,j}$ on
        \[
            \int_{\R^d} g(x) \sigma^{-n} c_k(\partial_{ij} k)(\sigma^{-1}(x-\theta)) \, \lambda(\mathrm{d}x)
        \]
        and thereby due to the factor $\sigma^{-2}$ and without loss of generality on $\partial_{\theta_i,\theta_j} \E_{\theta, \sigma}(g)$.
        \item By picking $\sigma^*_{\varepsilon}$, the largest of $(\sigma_{\varepsilon,i, j})_{i,j \in \N_{\leq n}}$, one obtains a uniform bound of $\varepsilon$ on all second partial derivatives of $\theta \in \R^n \mapsto \E_{\theta, \sigma^*_{\varepsilon}}(g)$.
    \end{itemize}
    It remains to be shown that if $\sigma$ is picked large enough, this implies strong convexity for an arbitrary parameter $m^*$ smaller than the convexity parameter of $r$, namely $m$.
    For all $x \in \R^n$ we want to have
    \[
        x^T (H(\E_{\cdot, \sigma}(r)) - m^*I_n + \varepsilon 1_n)x \geq 0 \,,
    \]
    where $H(\E_{\cdot, \sigma}(r))$ denotes the Hessian of $\E_{\cdot, \sigma}(r)$, $I_n$ is the identity, and $1_n$ is a matrix such that
    \[
        \abs{1_{n,ij}} \leq 1 \,\, \text{for all $i,j$.}
    \]
    $\varepsilon 1_n$ corresponds to the Hessian of $\E_{\cdot, \sigma}(g)$. We have
    \[
        x^T (H(\E_{\cdot, \sigma}(r)) - mI_n)x \geq 0
    \]
    by \cref{thm:strongconv} and, therefore, we show
    \begin{align*}
        x^T ((m-m^*)I_n + \varepsilon 1_n)x 
        &\geq
        (m-m^*) \max_{i} \abs{x_i}^2 - \varepsilon n^2 \max_i \abs{x_i}^2
        \\
        &=
        \max_{i} \abs{x_i}^2 (m-m^*- \varepsilon n^2)
        \geq
        0 \,.
    \end{align*}
    Which is true for $\varepsilon \leq \tfrac{m-m^*}{n^2}$.
\end{proof}

\cref{thm:fourier} suggests non-trivial sufficient conditions for functions on the reals whose stochastic relaxations turn out to be convex functions when paired with probability measures that have Schwartz densities. There are a number of handy results we obtain with \cref{thm:fourier}.

\begin{corollary}[Deterministic Cosine Disturbance]
\label{cor:cosinedist}
    Let $f=r+g: \R^n \to \R$ be the sum of a polynomially bounded, $m$-strongly convex function $r$ and for all $x\in\R^n$
    \[
        g(x) = \sum_{j = 1}^\infty a_j \cos(\langle \xi_j, x \rangle + \psi_j)
    \]
    with 
    \begin{itemize}
        \item $a_j \in \R$ with $\sum_j |a_j| < \infty$, and
        \item $\xi_j \in \R^n$ such that there exists $\varepsilon > 0$ with $\norm{\xi_j} \geq \varepsilon$ for all $j \in \N$.
    \end{itemize}
    Let $\Ps{k}$ be a probability measure with a rapidly decreasing Lebesgue density $k$ and $m^* < m$. Then, $f$ is $m^*$-strongly $\Ps{k}$-weakly convex.
\end{corollary}
\begin{proof}
We prove that $g$ fulfills $\mathcal{F}(g)_{|U_\varepsilon(0)} \equiv 0$ in the distributional sense and apply \cref{thm:fourier}.
We have for all $j \in \N$
\begin{align}
\label{eq:thm:cosinedist:1}
    \cos(\langle \xi_j, x \rangle + \psi_j)
    &=
    \frac{e^{i(\langle \xi_j, x \rangle + \psi_j)} + e^{-i(\langle \xi_j, x \rangle + \psi_j)}}{2}
\notag
    \\
    &=
    \frac{e^{i\psi_j}e^{i\langle \xi_j, x \rangle} + e^{-i\psi_j}e^{i\langle -\xi_j, x \rangle }}{2} \,.
\end{align}
Therefore, for all Schwartz functions $\varphi \in \mathcal{S}(\R^n)$, we get
\begin{align*}
    \mathcal{F}T_g(\varphi)
    &=
    T_g(\mathcal{F}\varphi)
    = \int_{\R^n} g \mathcal{F} \varphi \, \mathrm{d}\lambda
\tag*{\small{\text{(definition distributional Fourier transform)}}}
    \\
    &=
    \sum_{j=1}^\infty a_j \int_{\R^n} \cos(\langle \xi_j, x \rangle + \psi_j) \mathcal{F} \varphi \, \lambda(\mathrm{d}x)
\tag*{\small{\text{(definition $g$, linearity)}}}
    \\
    &=
    \sum_{j=1}^\infty \tfrac{a_j}{2}  \bigg(e^{i\psi_j} \int_{\R^n} e^{i\langle \xi_j, x \rangle} \mathcal{F} \varphi \, \lambda(\mathrm{d}x)
    + e^{-i\psi_j}\int_{\R^n} e^{i\langle -\xi_j, x \rangle} \mathcal{F} \varphi \, \lambda(\mathrm{d}x) \bigg)
\tag*{\small{\text{(\cref{eq:thm:cosinedist:1}, linearity)}}}
    \\
    &=
    \sum_{j=1}^\infty \tfrac{a_j}{2}  \left(e^{i\psi_j} \mathcal{F}^{-1}(\mathcal{F} \varphi)(\xi_j)  + e^{-i\psi_j}\mathcal{F}^{-1}(\mathcal{F} \varphi)(-\xi_j) \right)
\tag*{\small{\text{(definition inverse Fourier transform)}}}
    \\
    &=
    \sum_{j=1}^\infty \tfrac{a_j}{2}  \left(e^{i\psi_j} \varphi(\xi_j)  + e^{-i\psi_j} \varphi(-\xi_j) \right)
\tag*{\small{\text{(Fourier inversion theorem)}}}
\end{align*}
This implies $\mathcal{F}T_g = \sum_{j=1}^\infty \tfrac{a_j}{2}  \big(e^{i\psi_j} \delta_{\xi_j}  + e^{-i\psi_j} \delta_{-\xi_j} \big)$, where $\delta$ is the Dirac distribution. The assumption $\norm{\xi_j} > \varepsilon$ for all $j\in\N$ therefore implies $\mathcal{F}(g)_{|U_\varepsilon(0)} \equiv 0$.
\end{proof}

\cref{cor:cosinedist} also recovers the convexity result from \cref{ex:Rastrigin}.
\begin{corollary}
    The Rastrigin function
    \[
        x \in \R^n \longmapsto f(x) := \norm{x}^2 - \sum_{j=1}^n a_j \cos( \xi_j x_j) \,,
    \]
    where $\xi_j, a_j > 0$ for all $j \in \N_{\leq n}$, is $m^*$-strongly $\mathcal{N}(0, 1)$-weakly convex for all $m^*<1$.
\end{corollary}

We also obtain the following result for stochastic disturbances. From the algorithmic perspective, the following model constitutes static noise.

\begin{corollary}[Stochastic Cosine Disturbance]
    Let $f=r+g: \R^n \to \R$ be the sum of a polynomially bounded, $m$-strongly convex function $r$ and a random field
    \[
        g(x) = \sum_{j = 1}^\infty A_j \cos(\langle \xi_j, x \rangle + \psi_j)
    \]
    with 
    \begin{itemize}
        \item $(A_j)_{j\in\N}$ are real-valued random variables on the same probability space such that $\sum_j |A_j| < C$ with probability $p \in [0,1]$, and
        \item $\xi_j \in \R^n$ such that there exists $\varepsilon > 0$ with $\norm{\xi_j} \geq \varepsilon$ for all $j \in \N$.
    \end{itemize}
    Let $\Ps{k}$ be a probability measure with a rapidly decreasing Lebesgue density $k$ and $m^*<m$. Then, $f$ is $m^*$-strongly $\Ps{k}$-weakly convex with at least probability $p$ with the same threshold scaling constant $\sigma^*$. Further, $\sigma^*(C) \in O\big(\sqrt[d]{C}\big)$ for all $d \in \N$.
\end{corollary}
\begin{proof}
    Analogously to \cref{thm:fourier} and \cref{lem:fourier1} the result follows in conjunction with the following bound. For all Schwartz functions $\varphi \in \mathcal{S}(\R^n)$, we have with probability $p$
    \begin{align*}
        \mathcal{F}T_g(\varphi)
        &=
        \sum_{j=1}^\infty \tfrac{A_j}{2}  \left(e^{i\psi_j} \varphi(\xi_j)  + e^{-i\psi_j} \varphi(-\xi_j) \right)
    \tag*{\small{\text{(proof of \cref{cor:cosinedist})}}}
        \\
        \implies
        \abs{\mathcal{F}T_g(\varphi)}
        &\leq
        \sum_{j=1}^\infty \tfrac{\abs{A_j}}{2}  \left(
        \abs{\varphi(\xi_j)}  + \abs{\varphi(-\xi_j)} \right)
    \tag*{\small{\text{(triangle inequality)}}}
        \\
        &\leq
        \sum_{j=1}^\infty \abs{A_j}
        \sup_{x \in U_\varepsilon(0)^C} \abs{\varphi(x)} 
    \tag*{\small{\text{($\varphi$ is Schwartz)}}}
        \\
        &\leq C \sup_{x \in U_\varepsilon(0)^C} \abs{\varphi(x)} \,.
    \tag*{\small{\text{(assumption $\sum_j |A_j| < C$)}}}
    \end{align*}
    By the fact that $\varphi$ is rapidly decreasing (see \cref{lem:fourier1} for why $\varphi$ can be considered fixed), we also have
    \[
        C \sup_{x \in U_\varepsilon(0)^C} \abs{\varphi(\sigma x)} \leq C \tfrac{B}{\sigma^d}
    \]
    for some $B \geq 0$ and for all $\sigma > 0$ and for all $d \in \N$. Therefore, as we want to pick $\sigma^*$ to bound the value of $\abs{\mathcal{F}T_g(\sigma^*\varphi)}$ by say some $c>0$ dependent on the convexity parameter of $r$, in the spirit of \cref{thm:fourier} and \cref{lem:fourier1}, we have $\sigma^*(C) \in O\big(\sqrt[d]{C}\big)$ for all $d \in \N$ due to
    \[
        C \tfrac{B}{\sigma^d} = c \iff \sigma = \sqrt[d]{B/c} \sqrt[d]{C} \,.
    \]
\end{proof}

\section{Conclusion}

We established that stochastic relaxations have many favorable properties for optimization in surprisingly general settings. Concretely, we provided mild conditions under which optima are persistent under stochastic relaxation and thereby established the approach's consistency, which aligns with our intuition. Next to consistency, we have also shown that probability measures with Lipschitz differentiable densities under mild regularity conditions result in relaxations with Lipschitz differentials, allowing favorable convergence properties for first-order optimization methods. However, the arguably most important contributions are the convexity conditions for stochastic relaxations derived in the last section, which offer the possibility of principled global optimization under high-frequency, even stochastic, disturbances. Beyond that, we hope that the insight developed will be useful in furthering state-of-the-art optimization methods.

\paragraph{Future Work.} We believe that some of the most promising directions for future work lie in extending the convexity results for stochastic relaxations based on other methods from harmonic analysis. In practice, particularly in evolution strategies, the decoupling of the relaxation and (numerical) integration methods suggested by the setting at hand could offer the opportunity to design many new optimization methods. The gradient representation of relaxations of differentiable functions likely has powerful applications in problems where gradients are available, such as in gray-box problems in machine learning. Whether convexity results for relaxations hold for probability measures with a support on Lebesgue null sets may also pose an interesting question. These measures can offer strongly reduced computational cost in numerical integration---related algorithms have recently yielded spectacular results \cite{loshchilov2019lmmaes,choro2019asebo}. Connecting the convexity results with existing approaches to control the scaling of probability measures during an optimization process is also an important next step.
\newpage

\printbibliography
\end{document}